\documentclass[11pt]{amsart}

\usepackage[T1]{fontenc}
\usepackage[english]{babel}
\usepackage[utf8]{inputenc}

\usepackage{amsmath,amsthm,amssymb,amsfonts}
\usepackage{enumerate,amscd,amsxtra,MnSymbol}
\usepackage{mathrsfs}
\usepackage{bbm}

\usepackage{geometry}
\usepackage{color}
\usepackage[cmtip,all]{xy}
\usepackage{tikz} 
\usepackage{tikz-cd} 
\usepackage{tikz-qtree} 
\usepackage{forest}
\usetikzlibrary{arrows,calc}

\usepackage{graphicx}

\usepackage{xparse}

\usepackage[normalem]{ulem}

\usepackage{hyperref}
\usepackage[capitalise]{cleveref} 

\tikzstyle{d}=[double distance=.3ex]
\tikzstyle{w}=[preaction={draw=white, -,line width=4pt}]
\NewDocumentCommand{\punctuation}{ m m O{5pt} }{\node at ($(#1.east)-(0,#3)$) {#2};}

\newcounter{diagram}
\renewcommand{\thediagram}{\thetheorem}

\tikzset{%
node distance=1.5cm, la/.style={scale=0.8}, rr/.style={xshift=1.5cm},
space/.style={xshift=.5cm}, over/.style={auto=false,fill=white,inner sep=1.5pt, minimum size=0, outer sep=0},
    symbol/.style={%
        draw=none,
        every to/.append style={%
            edge node={node [sloped, allow upside down, auto=false]{$#1$}}},
            
    }, pro/.style={postaction={decorate,decoration={
        markings,
        mark=at position .5 with {\node at (0,0) {$\bullet$};}
      }},
      inner sep=.9ex,
      },
  n/.style={double equal sign distance, -implies}, t/.style={double distance=2.5pt, -implies, postaction={draw,-}},
}

\usepackage{mathtools}

\theoremstyle{definition}

\newtheorem{theorem}{Theorem}[section]
\newtheorem*{theorem*}{Theorem}
\newtheorem{lemma}[theorem]{Lemma}
\newtheorem{proposition}[theorem]{Proposition}
\newtheorem*{proposition*}{Proposition}
\newtheorem{corollary}[theorem]{Corollary}
\newtheorem*{corollary*}{Corollary}

\newtheorem{construction}[theorem]{Construction}

\newtheorem{introtheorem}{Theorem}

\newtheorem{definition*}{Definition}

\newtheorem{conjecture*}{Conjecture}

\newtheorem{notation}[theorem]{Notation}

\newtheorem{remark}[theorem]{Remark}
\newtheorem{remark*}{Remark}

\newtheorem{definition}[theorem]{Definition}

\newcommand{\caR}{{\mathcal R}}
\newcommand{\caC}{{\mathcal C}}
\newcommand{\caD}{{\mathcal D}}
\newcommand{\caB}{{\mathcal B}}
\newcommand{\caA}{{\mathcal A}}
\newcommand{\caE}{{\mathcal E}}

\newcommand{\caM}{{\mathcal M}}
\newcommand{\caN}{{\mathcal N}}

\newcommand{\caH}{{\mathcal H}}

\newcommand{\caP}{{\mathcal P}}
\newcommand{\caK}{{\mathcal K}}

\newcommand{\caF}{{\mathcal F}}
\newcommand{\caV}{{\mathcal V}}

\newcommand{\D}{{\mathsf D}}
\renewcommand{\L}{{\mathbb L}}

\newcommand{\op}{\mathsf{op}}
\newcommand{\dual}{\vee}

\newcommand{\Sp}{\mathsf{Sp}}

\newcommand{\y}{\mathsf{y}}

\newcommand{\colim}{\mathrm{colim}}

\newcommand{\Mod}{{\mathsf{Mod}}}
\newcommand{\LMod}{{\mathsf{LMod}}}

\newcommand{\Alg}{{\mathsf{Alg}}}

\newcommand{\Y}{\mathrm{Y}}

\newcommand{\id}{\mathsf{id}}

\newcommand{\Cat}{\mathsf{Cat}}

\newcommand{\map}{\mathrm{map}}

\renewcommand{\Pr}{\mathsf{Pr}}

\newcommand{\Spc}{\mathsf{Spc}}

\newcommand{\Add}{\mathsf{Add}}

\newcommand{\Fun}{\mathsf{Fun}}

\newcommand{\add}{\mathsf{add}}

\newcommand{\perf}{\mathsf{perf}}

\newcommand{\Ar}{\mathsf{Ar}}

\newcommand{\KR}{\mathsf{KR}}

\newcommand{\Wald}{\mathsf{Wald}}

\newcommand{\fiss}{\mathsf{fiss}}
\newcommand{\gd}{\mathsf{gd}}
\newcommand{\Exact}{\mathsf{Exact}}
\renewcommand{\d}{\mathsf{d}}

\newcommand{\Exc}{\mathsf{Exc}}

\newcommand{\ev}{\mathsf{ev}}

\newcommand{\Calg}{\mathsf{CAlg}}
\newcommand{\exc}{\mathsf{exc}}

\newcommand{\triv}{{\mathrm{triv}}}

\usepackage[textwidth=3cm,
textsize=small,
colorinlistoftodos]
{todonotes} 

\title{Real $K$-theory for Waldhausen $\infty$-categories with genuine duality}

\author[H.\ Heine]{Hadrian Heine}
\address{Max Planck Institute f\"ur Mathematik, Bonn, Germany}
\email{heine@mpim-bonn.mpg.de}

\author[M.\ Spitzweck]{Markus Spitzweck}
\address{Fakult\"at f\"ur Mathematik, Universit\"at Osnabr\"uck, Osnabr\"uck, Germany}
\email{markus.spitzweck@uni-osnabrueck.de}

\author[P.\ Verdugo]{Paula Verdugo}
\address{Max Planck Institute f\"ur Mathematik, Bonn, Germany}
\email{verdugo@mpim-bonn.mpg.de}

\title{A multiplicative universal property for real algebraic $K$-theory}

\begin{document}

\begin{abstract}
We prove that the real $K$-theory functor is lax symmetric monoidal.  Moreover, it is initial among those real functors $\Wald_\infty^\gd\to\Sp^{C_2}$ that are additive and lax symmetric monoidal.
\end{abstract}

\maketitle
\setcounter{tocdepth}{1}
\tableofcontents

\section{Introduction}

At its core, with several different incarnations, algebraic $K$-theory takes an algebraic-flavored input and associates to it a spectrum, $K\colon\caC\to\Sp$. Since $\Sp$ has a canonical monoidal structure given by the smash product $\otimes=\wedge$, it is natural to wonder whether the functor of $K$-theory can be made lax symmetric monoidal. 

This question has a long history, in \cite{elmendorf-mandell} Elmendorf and Mandell upgraded the existent algebraic $K$-theory of symmetric monoidal categories \cite{May_loop_spaces,Segal_categories_cohomology} and permutative categories \cite{mandell_inverseK,gurski-johnson-osorno-Ktheory_2cats} to the context of multicategories, where it can be shown to be lax monoidal. Moreover, a multiplicative version of Waldhausen's $S$-construction also in the context of multicategories was studied in \cite{inna_multiplicative}.

Naturally, this question remained active for hermitian/real versions of algebraic $K$-theory. For hermitian $K$-theory, Schlichting has proven it whenever $2$ is invertible in the base ring \cite[Section 5]{schlichting.derived}. For \emph{real} algebraic $K$-theory, the reference is by Hesselholt and Madsen \cite{hesselholt-madsen}.

In the last fifteen years, in independent work, Barwick in \cite{barwick.wald} and Blumberg-Gepner-Tabuada in \cite{bgt} have built frameworks for algebraic $K$-theory for $\infty$-categories, which provided new and powerful tools to the area. 

As part of this, they addressed and upgraded the multiplicativity structure of algebraic $K$-theory. Blumberg-Gepner-Tabuada adapted Quillen's $Q$-construction to stable $\infty$-categories and showed that not only algebraic $K$-theory is lax symmetric monoidal, but moreover it is initial among additive and multiplicative functors when restricting the source to idempotent complete stable $\infty$-categories.  Barwick, in turned, built $K$-theory using $\infty$-categorical versions of the $S$-construction and Waldhausen categories, and showed such universal property among additive and multiplicative functors on all Waldhausen $\infty$-categories. 

In this work, we settle these questions for real algebraic $K$-theory in the $\infty$-context. Specifically, in the framework developed in \cite{hsv,comparisonKth} through a refinement of the $S$-construction for Waldhausen $\infty$-categories with genuine duality. From now on, the word ``real'' here is largely used to mean enriched in the $\infty$-category of genuine $C_2$-spaces, $\Spc^{C_2}$.

\begin{introtheorem}[\cref{thm:univ_prop_Sp}]
The real $K$-theory functor $\KR\colon\Wald^\gd_\infty\to \Sp^{C_2}$ is lax symmetric monoidal.
\end{introtheorem}
In \cite{Calmes_etal1} they show that their hermitian $K$-theory functor is lax symmetric monoidal.

One of the motivations for studying the multiplicative structure of algebraic $K$-theory is to obtain a more structured output through ring spectra. We prove this phenomenon for real $K$-theory.

\begin{introtheorem}[\cref{cor:normed_Einfinty_ringspectrum}]
Let $A$ be a normed $\mathbb{E}_\infty$-ring spectrum.
Then  $$\KR(A) \coloneqq\KR(\Mod_A(\Sp)^\perf)$$ is an $\mathbb{E}_\infty$-ring genuine $C_2$-spectrum, where the genuine refinement on $\Mod_A(\Sp)^\perf$ is the genuine symmetric one. In consequence, we obtain a functor $$ \KR\colon \Alg_{\mathbb{N}_\infty}(\Sp^{C_2}) \to  \Alg_{\mathbb{E}_\infty}(\Sp^{C_2}).$$
\end{introtheorem}

In this work we further prove a universal property for $\KR$, providing a real version of the universal property of \cite{barwick.mult} in the case of (non-real) algebraic $K$-theory.

\begin{introtheorem}[Universal property, \cref{thm:univ_prop_Sp}]\label{intro_universal}
The real $K$-theory functor $\KR\colon\Wald^\gd_\infty\to \Sp^{C_2}$ is initial among those real functors $\Wald^\gd_\infty\to \Sp^{C_2}$ that are additive and lax symmetric monoidal.
\end{introtheorem}

The universal property of \cref{intro_universal} provides a strong tool to study real trace methods. Indeed, we know that given any lax symmetric monoidal real theory, it will be the target of exactly one map out of $\KR$. Furthermore, the proofs provide access to a fairly explicit description of this map.

\subsubsection*{A multiplicative story for non-commutative motives} Following \cite{barwick.wald}, we had shown in \cite{hsv} that the real $K$-theory functors admits a factorization
\[
\begin{tikzcd}
\Wald_\infty^\gd\ar[rr, "\KR"]\ar[dr]        &&\Spc^{C_2}\\
&\caM\ar[ur, "\KR'"']  &
\end{tikzcd}
\]
where $\caM$ may be interpreted as the real $\infty$-category of non-commutative motives, and $\KR'$ is the genuine excisive approximation of the canonical real functor $\iota\colon \caM \to \Spc^{C_2}$ that sends $\caC\in \caM$ to its space of objects. 

\begin{introtheorem}[\cref{Dwaldfiss_monoidal}]
The $\infty$-categories of non-commutative motives $\caM$ admits a real symmetric monoidal structure compatible with that of $\Wald_\infty^\gd$. 
\end{introtheorem}

\begin{introtheorem}[\cref{motive_lax_sym,prop:KR'_initial}]\label{intro_motive_universal_prop}
The real funcor $\KR'\colon \caM\to \Sp^{C_2}$ is lax symmetric monoidal. Moreover, it is initial among left adjoint and genuine excisive lax symmetric monoidal real functors $\caM\to\Sp^{C_2}$.
\end{introtheorem}

\section{Monoidal structure on $\Wald_\infty^\gd$}
\subsection{Waldhausen $\infty$-categories with genuine duality}

\begin{definition}
A Waldhausen $\infty$-category is a pair $(\caD, \caC) $ with $\caD$ an  $\infty$-category with zero object and $\caC$ a subcategory of $\caD$, whose morphisms are called cofibrations, subject to the following conditions.

\begin{itemize}
\item Every object $X$ of $\caD$ is cofibrant, i.e. the unique map $0 \to X$ is a cofibration.
	
\item The cobase change in $\caD$ of any cofibration exists and is again a cofibration.
\end{itemize}

A map of Waldhausen $\infty$-categories $(\caD, \caC) \to (\caD', \caC') $
is a functor $\caD \to \caD'$ that preserves the initial object, cofibrations and pushouts along cofibrations.
\end{definition}

We denote by $\Wald_\infty \subset \Fun([1], \Cat_\infty)$ the subcategory spanned by the Waldhausen $\infty$-categories and maps between them as above. Note that the evaluation at the target defines a faithful functor $\Wald_\infty \to \Cat_\infty$, i.e.\ a functor that yields embeddings of spaces on mapping spaces.

The notion of Waldhausen $\infty$-category is not self dual, and therefore does not lend itself to easily take homotopy fixed points which is a central point in our work, as we want to consider dualities. However, we can combine two different Waldhausen structures and define the notion of \emph{exact} $\infty$-category, first introduced in \cite{barwick.exact}.

\begin{definition}A pair $(\caC,\caF)$ is a co-Waldhausen $\infty$-category if $(\caC^\op,\caF^\op)$ is a Waldhausen $\infty$-category. 
\end{definition}

 An exact $\infty$-category will consist of an $\infty$-category $\caD$ compatibly carrying a Waldhausen and co-Waldhausen structure, where the compatibility is encoded in the following notion introduced in \cite{barwick.exact}.

\begin{definition}Let $(\caD,\caC,\caF)$ be a triple where $\caD$ is an $\infty$-category and $\caC,\caF$ are subcategories, whose morphisms we call cofibrations and fibrations respectively. A commutative square $\sigma$ in $\caD$
\[
\begin{tikzcd}
A \ar[r, "f"] \ar[d, "\alpha"'] & B \ar[d, "\beta"] \\
C \ar[r, "g"']  & D
\end{tikzcd}
\]

\begin{itemize}
\item is an ambigressive pushout square if $\sigma$ is a pushout square, $f$ is a cofibration and $\alpha$ is a fibration, and 
\item an ambigressive pullback square if $\sigma$ is a pullback square, $g$ is a cofibration and $\beta$ is a fibration.
\end{itemize}

We call a square as above exact when is both an ambigressive pushout and an ambigressive pullback.
\end{definition}

\begin{definition}\label{def:exact_infty_cats}
An exact $\infty$-category is a triple $(\caD, \caC, \caF)$ with $\caD$ an additive $\infty$-category and $\caC, \caF$ subcategories of $\caD$, whose morphisms are called respectively cofibrations and fibrations, subject to the following conditions.

\begin{enumerate}
\item The pair $(\caD, \caC) $ is a Waldhausen $\infty$-category.
\item The pair $(\caD, \caF) $ is a co-Waldhausen $\infty$-category.
\item A commutative square in $\caD$ is an ambigressive pushout square if and only if it is an ambigressive pullback square.
\end{enumerate}
\end{definition}

The morphisms between them are the natural ones, as given in the following definition.

\begin{definition}
Given exact $\infty$-categories $(\caD, \caC, \caF), (\caD', \caC', \caF') $
a functor $F\colon \caD \to \caD'$ is called exact (with respect to the exact structures on $\caC$ and $\caD$) if $F$ preserves the zero object, cofibrations, fibrations, pushouts along cofibrations and pullbacks along fibrations.

\end{definition}

In light of condition 3, in an exact $\infty$-category we call ambigressive pushout squares respectively ambigressive pullback squares simply ambigressive squares.

\begin{remark}
A functor $F\colon \caD \to \caD'$ is exact if and only if $F$ preserves the zero object, cofibrations and cofibers of cofibrations or dually the zero object, fibrations and fibers of fibrations.
\end{remark}

Denote by $\Exact_\infty \subset \Fun(\Lambda^2_2, \Cat_\infty)$ the subcategory spanned by the exact $\infty$-categories and exact functors.
Note that valuation at $2 \in \Lambda^2_2$ defines a faithful functor $\Exact_\infty \to \Cat_\infty$.\ Restriction along the embedding $[1] \subset \Lambda^2_2$ that sends 0 to 0 and 1 to 2 defines a functor $$\Fun(\Lambda^2_2, \Cat_\infty) \to \Fun([1], \Cat_\infty)$$ that restricts to a fully faithful functor $$\Exact_\infty \to \Wald_\infty.$$

In the next result, we characterize the additivity condition on an $\infty$-category satisfying all conditions of being exact, except being additive.

\begin{proposition}
Let $\caD$ be an $\infty$-category that admits a zero object and $\caC,\caF$ subcategories of $\caD$, whose morphisms we call cofibrations and  fibrations respectively. If conditions (1)-(3) of \cref{def:exact_infty_cats} hold, then $\caD$ is additive if and only if the following condition holds:
a commutative square 
\begin{equation*}\label{jjj1}
\begin{tikzcd}
X \ar[r] \ar[d]     & Y \ar[d]  \\
Z \ar[r]            & W
\end{tikzcd}
\end{equation*}
in $\caD$, where both vertical maps are fibrations, is a pullback square if it induces an equivalence on fibers of the vertical morphisms.
\end{proposition}

\begin{remark}The embedding $\Exact_\infty \subset \Wald_\infty$ is a localization. Indeed, we know that $\Exact_\infty$ is presentable, and that the $\infty$-categories $\Wald_\infty \subset \Fun([1], \Cat_\infty), \ \Exact_\infty \subset \Fun(\Lambda^2_2, \Cat_\infty)$ are closed under small limits and small filtered colimits, and therefore $\Exact_\infty$ is closed  under small limits and small filtered colimits in $\Wald_\infty$. 
\end{remark}

In order to define a duality on Waldhausen $\infty$-categories, we note that the non-trivial $C_2$-actions on $\Lambda^2_2$ and $\Cat_\infty$ induce a 
$C_2$-action on $\Fun(\Lambda^2_2, \Cat_\infty)$ that restricts to $\Exact_\infty$. This allows us to introduce the following definition.

\begin{definition}
We define the $\infty$-category of small Waldhausen $\infty$-categories with duality as the homotopy fixed points of the action on $\Exact_\infty$ described above. We write $$\Wald_\infty^\d\coloneqq \Exact_\infty^{hC_2}.$$
\end{definition}

In other words, a Waldhausen $\infty$-category with duality is a Waldhausen $\infty$-category, whose underlying $\infty$-category carries a duality such that the cofibrations together with the opposites of the cofibrations
form an exact $\infty$-category. 

The next step is to refine this notion of duality, for which we define Waldhausen $\infty$-categories with \emph{genuine duality}. A key component of the definition, that we won't introduce here, is the hermitian objects functor $\mathcal{H}\colon \Cat_\infty^\mathrm{d}\to \Cat_\infty$ which plays the role of homotopy fixed points, defined in \cite[Section 2.3]{hsv}.

\begin{definition}\label{Waldgd}
A small Waldhausen $\infty$-category with genuine duality is 
a pair $(\caE, \phi)$, with $E$ a small Waldhausen $\infty$-category with duality
and $\phi\colon  H \to \caH(\caE)$ a right fibration enjoying the following conditions.

\begin{enumerate}
\item $(\caE, \phi)$ is an additive $\infty$-category with genuine duality.
\item for every commutative square
\begin{equation*}
\begin{tikzcd}
A \ar[r, "f"] \ar[d, "\alpha"']         &B \ar[d, "\beta"] \\
	C \ar[r, "g"']                      &D
\end{tikzcd}
\end{equation*}
in $\caH(\caE)$ lying over a pushout square in $\caE$ such that $f$ lies over a cofibration, the induced square of spaces
\[
\begin{tikzcd}
H(D) \ar[r] \ar[d]          &H(B) \ar[d] \\
H(C) \ar[r]                 &H(A)
\end{tikzcd}
\]
is a pullback square.
\end{enumerate}
\end{definition}

\begin{definition}
We define the $\infty$-category of small Waldhausen $\infty$-categories with genuine duality $\Wald_\infty^\gd$ as the full subcategory of $\Wald_\infty^\d \times_{\Cat_\infty^{hC_2}} \Cat_\infty^\gd$ spanned by the small Waldhausen $\infty$-categories with genuine duality. 
\end{definition}

\begin{remark}
By the pullback definition of $\Cat_\infty^\gd$ and the pullback pasting law, there is a canonical equivalence 
$$ \Wald_\infty^\gd \simeq \Wald_\infty^\d \times_{\Cat_\infty} \caR,$$
where the pullback is taken along the functor $\caR \to \Cat_\infty $ evaluating at the target and the functor $\caH\colon \Wald_\infty^\d \to \Cat_\infty $
taking hermitian objects.
\end{remark}

\begin{remark}
The $\infty$-category $\Wald_\infty^\gd$ is $\aleph_1$-compact since
the $\infty$-categories $\Wald_\infty^\d , \Cat_\infty, \caR $
are compactly generated and the functor $\caR \to \Cat_\infty $ evaluating at the target preserves compact objects and the functor $\caH\colon \Wald_\infty^\d \to \Cat_\infty $ preserves $\aleph_1$-compact objects.
\end{remark}

\subsection{Brief comments on real structures} Before we move on to the monoidal structure, we recall and fix terminology that will be used throughout the work originally developed in \cite[Section 3]{hsv}. 

\begin{definition} An $\infty$-category $\caC$ is a \emph{real} $\infty$-category if it is enriched in the $\infty$-category of genuine $C_2$-spaces, $\Spc^{C_2}$. A \emph{real} functor between real $\infty$-categories is a $\Spc^{C_2}$-enriched functor.
\end{definition}

\begin{definition}
A real symmetric monoidal $\infty$-category is a commutative monoid in $\Spc^{C_2}\mathrm{-}\Cat$, the $\infty$-category of real $\infty$-categories.
\end{definition}

\begin{definition}
A presentably real symmetric monoidal $\infty$-category is a presentable real $\infty$-category whose tensor product is componentwise a real left adjoint.
\end{definition}

\begin{remark}
A presentably real symmetric monoidal $\infty$-category is likewise a commutative algebra in $\Mod_{\Spc^{C_2}}(\Pr^\L)$, where $\Pr^L$
is the symmetric monoidal $\infty$-category of presentable $\infty$-categories and left adjoint functors and $\Mod_{\Spc^{C_2}}(\Pr^\L)$ is the $\infty$-category of $ \Spc^{C_2}$-modules in $\Pr^\L$.
This description implies that a presentably real symmetric monoidal $\infty$-category is likewise a presentably symmetric monoidal $\infty$-category $\caC$ equipped with a left adjoint symmetric monoidal functor $\Spc^{C_2} \to \caC.$
\end{remark}

\subsection{Monoidal structure on $\Wald_\infty^\gd$}
In \cite{hsv} we have built a monoidal structure on $\Wald_\infty^\gd$, which we used to show that it is a real $\infty$-category \cite[Corollary 7.18]{hsv}. We recall the construction here.

\begin{proposition}[{\cite[Proposition 7.10.]{hsv}\label{dfghjkgfl}}]
The $\infty$-category $\Exact_\infty$ carries a closed symmetric monoidal structure with $C_2$-action and the free functor $\Cat_\infty \to \Exact_\infty$ refines to a $C_2$-equivariant symmetric monoidal functor.
\end{proposition}

\begin{corollary}[{\cite[Corollary 7.11.]{hsv}}]\label{dfgkgfl}
The $\infty$-category $\Wald_\infty^\d$ carries a closed symmetric monoidal structure and the free functor $\Cat_\infty^{hC_2} \to \Wald_\infty^\d$ refines to a symmetric monoidal functor.
\end{corollary}

\begin{remark}\label{description_tensor_Waldd}
These monoidal structures are closely related. A description of the one in $\Exact_\infty$ can be found in \cite[Proposition 7.10]{hsv}. Moreover, its tensor unit is the free exact $\infty$-category on one generator, which is the full subcategory of the $\infty$-category of spectra $\Sp$ spanned by the finite sums of sphere spectra. 

Now, for $\Wald_\infty^\d$, given $\caC$ and $\caD$ Waldhausen $\infty$-categories with duality, its tensor product is given by the tensor product $\caC \otimes \caD$ in $\Exact_\infty$
equipped with the induced duality
$$(\caC \otimes \caD)^\op \simeq \caC^\op \otimes \caD^\op \simeq \caC \otimes \caD. $$

Moreover, the tensor unit of $\Wald^\d_\infty$ is the the free exact $\infty$-category on one generator equipped with the restriction of the Spanier-Whitehead duality.
\end{remark}

\begin{remark}\label{rmk:Waldgd_subcat_pullback} We can think of $\Wald^\gd_\infty$ as the full subcategory  spanned the Waldhausen $\infty$-categories with genuine duality of the pullback below 

\[
\begin{tikzcd}
\caP \ar[r] \ar[d]       & \Wald_\infty^{d} \ar[d, "\caH"]\\
\caR \ar[r, "\mathrm{ev}_1"']     & \Cat_\infty,
\end{tikzcd}
\]
where $\caR \subset \Fun([1], \Cat_\infty)$ is the full subcategory of right fibrations in the sense of \cite[Definition 2.24]{hsv}, $\ev_1$ is the cocartesian fibration that evaluates at the target, and $\caH$ is the functor that takes hermitian objects defined in \cite[Definition 2.29]{hsv}.
\end{remark}

From the interpretation in \cref{rmk:Waldgd_subcat_pullback}, it follows  that there is a canonical cocartesian fibration $\Wald^\gd_\infty\to \Wald^\d_\infty$. As we showed in \cite[Section 7]{hsv}, the closed symmetric monoidal structure on $\Wald_\infty^\d$ lifts to a closed symmetric monoidal structure on $\Wald_\infty^\gd$ along this functor.

\begin{proposition}[{\cite[Proposition 7.17.]{hsv}}]\label{fhbhnj}
The $\infty$-category $\Wald_\infty^\gd $ carries a closed symmetric monoidal structure and the forgetful functor $\Wald_\infty^\gd \to \Wald_\infty^\d$ is a cocartesian fibration of symmetric monoidal $\infty$-categories.
Moreover the free functor $ \Cat_\infty^\gd \to \Wald_\infty^\gd $ is a symmetric monoidal functor.
\end{proposition}

\begin{remark}\label{description_tensor_Waldgd} Explicitly, given two Waldhausen $\infty$-categories with genuine duality 
$(\caC, \phi\colon H \to \caH(\caC))$ and $ (\caD, \psi\colon T \to \caH(\caD))$, its tensor product 
 $$(\caC, \phi\colon H \to \caH(\caC)) \otimes (\caD, \psi\colon T \to \caH(\caD)) $$
 is given by the tensor product $ \caC \otimes \caD$ in $\Wald_\infty^d$ and the free Waldhausen structure with genuine duality on $\caC \otimes \caD$ generated by the free right fibration on the functor 
$$H \times T \to  \caH(\caC) \times  \caH(\caD) \simeq \caH(\caC \times \caD) \to \caH(\caC \otimes \caD).$$

The tensor unit of $\Wald_\infty^\gd $ is the tensor unit of $\Wald_\infty^\d$ equipped with the standard genuine refinement.
\end{remark}

\section{Monoidal structure on $\Sp^{C_2}$}

We begin by specializing the theory of enriched spectra explained in \cite[Section 10.2]{hsv} to our case of interest: genuine $C_2$-spectrum. Note, however, that the proofs of this section hold in full generality for enriched spectra as defined therein.

\begin{notation}
    We write $S^{2,1}$ for the $2$-sphere with sign representation.
\end{notation}

\begin{notation}
Denote by $(\Spc_*^{C_2})_{S^{2,1}} \subset \Spc_*^{C_2}$ the smallest symmetric monoidal full subcategory containing $S^{2,1}$.
\end{notation}

\begin{definition}\label{def:V_enriched_K_spectra}
The reduced real $\infty$-category of genuine $C_2$-spectra is the full real subcategory 
$$\Sp^{C_2} \subset \Fun_{\Spc_*^{C_2}}((\Spc_*^{C_2})_{S^{2,1}}, \Spc_*^{C_2})$$
spanned by the reduced real functors $F: (\Spc_*^{C_2})_{S^{2,1}} \to \Spc_*^{C_2}$ which are $S^{2,1}$-excisive, i.e.\ for every $X \in (\Spc_*^{C_2})_{S^{2,1}}$
the canonical morphism $F(X) \to F(S^{2,1}\wedge X)^{S^{2,1}}$ is an equivalence.
\end{definition}

\begin{remark} For the interested reader, or that who wishes the write the following proofs in full generality, we redirect them to see genuine $C_2$-spectra as a special case of enriched spectra as in \cite[Definition 10.26 and Notation 10.27]{hsv} where it is 

$$\Sp^{C_2}\coloneqq \Sp^{\Spc_\ast^{C_2}}_{S^{2,1}}(\Spc^{C_2}_\ast).$$

That is, we consider $\caV=\Spc^{C_2}_\ast$ and $K=S^{2,1}$.
\end{remark}

\subsection{Monoidal structure on genuine $C_2$-spectra}

\begin{proposition}[{\cite[Remark 10.11]{hsv}}]
An object of $\Fun_{\Spc_*^{C_2}}((\Spc_*^{C_2})_{S^{2,1}}, \Spc_*^{C_2})$ belongs to $\Sp^{C_2}$ if and only if it is local with respect to the set of morphisms
$$ \{ \mathrm{lan}_{S^{2,1} \wedge Y}(S^{2,1} \wedge Z) \to \mathrm{lan}_{Y}(Z) \mid Y \in (\Spc_*^{C_2})_{S^{2,1}}, \ Z \in \Spc_*^{C_2} \} .$$

Hence the real embedding $$\Sp^{C_2} \subset \Fun_{\Spc_*^{C_2}}((\Spc_*^{C_2})_{S^{2,1}}, \Spc_*^{C_2})$$ admits a real left adjoint.
\end{proposition}

\begin{proof}
An object $X \in \Fun_{\Spc_*^{C_2}}((\Spc_*^{C_2})_{S^{2,1}}, \Spc_*^{C_2})$
is local with respect to the set $$ \{ \mathrm{lan}_{S^{2,1} \wedge Y}(S^{2,1} \wedge Z) \to \mathrm{lan}_{Y}(Z) \mid Y \in (\Spc_*^{C_2})_{S^{2,1}}, \ Z \in \Spc_*^{C_2} \} $$
if and only if the following map
$$ \map_{ \Fun_{\Spc_*^{C_2}}((\Spc_*^{C_2})_{S^{2,1}}, \Spc_*^{C_2})}(\mathrm{lan}_{Y}(Z),X) \to $$$$ \map_{ \Fun_{\Spc_*^{C_2}}((\Spc_*^{C_2})_{S^{2,1}}, \Spc_*^{C_2})}(\{ \mathrm{lan}_{S^{2,1} \wedge Y}(S^{2,1} \wedge Z),X) $$
is an equivalence. The latter map identifies with the following map
$$ \map_{ \Spc_*^{C_2}}(Z,X(Y)) \to \map_{\Spc_*^{C_2}}(Z, X(S^{2,1} \wedge Y)^{S^{2,1}}). $$
Hence by Yoneda, this holds if and only if the canonical morphism $ X(Y) \to  X(S^{2,1} \wedge Y)^{S^{2,1}}$ is an equivalence.
\end{proof}

\begin{definition}
Let $\caC$ be a reduced real $\infty$-category that admits tensors with $S^{2,1}$
and $\caD$ a reduced real $\infty$-category that admits cotensors with $S^{2,1}$.
A reduced real functor $F\colon \caC \to \caD$ is $S^{2,1}$-excisive if for every $X \in \caC$ the canonical morphism $F(X) \to F(S^{2,1}\wedge X)^{S^{2,1}}$ is an equivalence.
\end{definition}

\begin{notation}
Let $\caC$ be a reduced real $\infty$-category that admits tensors with $S^{2,1}$
and $\caD$ a reduced real $\infty$-category that admits cotensors with $S^{2,1}$.
Let $\Exc_{S^{2,1}}(\caC,\caD) \subset  \Fun_{\Spc_*^{C_2}}(\caC, \caD)$
be the full real subcategory of $S^{2,1}$-excisive functors $\caC \to \caD.$
\end{notation}

\begin{proposition}\label{prop:localization_is_monoidal}
Let $\caC$ be a small symmetric monoidal real $\infty$-category and
$\caD$ a presentably symmetric monoidal real $\infty$-category.
The real localization $$ \Fun_{\Spc_*^{C_2}}(\caC, \caD) \leftrightarrows \Exc_{S^{2,1}}(\caC,\caD)$$ is compatible with the symmetric monoidal structure on $\Fun_{\Spc_*^{C_2}}(\caC, \caD)$ given by the Day convolution. That is, tensoring local equivalences yields a local equivalence.
Hence $\Sp^{C_2}$ carries a presentably symmetric monoidal structure.
\end{proposition}

\begin{proof}
It suffices to see that $$ \mathrm{lan}_{S^{2,1} \wedge Y}(S^{2,1} \wedge Z) \otimes \mathrm{lan}_{Y'}(Z') \to \mathrm{lan}_{Y}(Z) \otimes \mathrm{lan}_{Y'}(Z') $$ is a local equivalence for every $$ Y, Y' \in (\Spc_*^{C_2})_{S^{2,1}}, Z,Z' \in \Spc_*^{C_2}. $$

This is true since the latter morphism identifies with the following one: 
$$ \mathrm{lan}_{S^{2,1} \wedge Y \otimes \Y' }(S^{2,1} \wedge Z \otimes Z') \to \mathrm{lan}_{Y \otimes Y'}(Z \otimes Z').$$
\end{proof}

\begin{corollary}\label{cor:localization_is_monoidal}
The real localization $$ \Fun_{\Spc_*^{C_2}}((\Spc_*^{C_2})_{S^{2,1}}, \Spc_*^{C_2}) \leftrightarrows \Sp^{C_2}$$ is compatible with the symmetric monoidal structure. That is, tensoring local equivalences yields a local equivalence.
Hence $\Sp^{C_2}$ carries a presentably symmetric monoidal structure.
\end{corollary}

\begin{remark}\label{description:monoidal_str_enriched_spectra}
The tensor unit of $\Sp^{C_2}$ is the real functor $(\Spc_*^{C_2})_{S^{2,1}} \to \Spc_*^{C_2}$ given  by 
$$X \mapsto \colim_{n \geq 0}(\Omega^{S^{2,1}})^n(\Sigma^{S^{2,1}})^n(X).$$

The tensor product of two real functors $F,G \in \Sp^{C_2}$ is the real functor $ F \otimes G\colon (\Spc_*^{C_2})_{S^{2,1}} \to \Spc_*^{C_2}$ given  by 
$$X \mapsto \colim_{n \geq 0}(\Omega^{S^{2,1}})^n(\Sigma^{S^{2,1}})^n(\wedge_!(\wedge \circ (F \times G))),$$
where $\wedge_!$ is the left Kan extension along $\wedge$.
\end{remark}

\begin{definition} We define the real infinity loop functor $\Omega^\infty\colon \Sp^{C_2} \to \Spc_*^{C_2}$ to be the real functor
$$ \Sp^{C_2} \subset \Fun_{\Spc_*^{C_2}}((\Spc_*^{C_2})_{S^{2,1}}, \Spc_*^{C_2}) \to \Spc_*^{C_2} $$
evaluating at the tensor unit of $\Spc_*^{C_2}$, which is $S^{0,0}$, the set with two elements and the trivial $C_2$-action.
\end{definition}

\begin{remark}
The real functor $\Omega^\infty\colon \Sp^{C_2} \to \Spc_*^{C_2}$ is lax symmetric monoidal. Indeed, we can consider its real left adjoint $\Sigma^\infty\colon \Spc_*^{C_2} \to \Sp^{C_2} $, which factors as the real functor $$\Spc_*^{C_2} \to \Fun_{\Spc_*^{C_2}}((\Spc_*^{C_2})_{S^{2,1}}, \Spc_*^{C_2}) $$ assigning the left Kan extension along the real functor $* \to (\Spc_*^{C_2})_{S^{2,1}}$ taking the tensor unit,
followed by the real localization functor $$ \Fun_{\Spc_*^{C_2}}((\Spc_*^{C_2})_{S^{2,1}}, \Spc_*^{C_2}) \to \Sp^{C_2}.$$

Now, the real functor $\Sigma^\infty\colon \Spc_*^{C_2} \to \Sp^{C_2} $ is symmetric monoidal because both of these maps are so, and therefore $\Omega^\infty$ is lax symmetric monoidal.
\end{remark}

The following result provides a multiplicative universal property of enriched spectra. Before that, we need to briefly introduce real excisive functors. 

\begin{definition}Let $F\colon\caC\to\caD$ be a reduced real functor between pointed real $\infty$-categories, where $\caC$ admits tensors with $S^{1,1}$ as $\Spc^{C_2}_\ast$-enriched $\infty$-category and $\caD $ admits cotensors with $S^{1,1}$ as $\Spc^{C_2}_\ast$-enriched $\infty$-category. We say that $F\colon \caC \to \caD$ is genuine excisive if for every $X \in \caC$ the natural map below is an equivalence $$ F(X) \to \Omega^{1,1}(F(\Sigma^{1,1}(X))).$$
\end{definition}

\begin{notation}Let $\caC$ be a small symmetric monoidal real $\infty$-category and
$\caD$ a presentably symmetric monoidal real $\infty$-category. We denote by $$\Exc^{\otimes}_{S^{2,1}}(\caC,\caD)$$
the $\infty$-category of commutative algebras in $\Exc_{S^{2,1}}(\caC,\caD).$

\end{notation}

\begin{theorem}\label{genspec}
Let $\caC$ be a small real symmetric monoidal $\infty$-category. Suppose that the tensor with $S^{2,1}\in \Spc^C_2$ exists in $\caC$ and that the cotensor with $S^{2,1}$ exists in $\caD$. Then the real functor $\Omega^\infty \colon \Sp^{C_2} \to \Spc^{C_2}_\ast$ induces a symmetric monoidal equivalence $$\Exc_{S^{2,1}}(\caC,\Sp^{C_2}) \to \Exc_{S^{2,1}}(\caC,\Spc^{C_2}_\ast).$$
\end{theorem}

\begin{proof}
By \cref{prop:localization_is_monoidal} we have a symmetric monoidal real functor
$$\Fun^{\Spc^{C_2}_*}(\Spc^{C_2}_{S^{2,1}},\Spc^{C_2}) \to \Sp^{C_2}$$
left adjoint to the real embedding. It gives rise to a symmetric monoidal real functor 
$$\Sigma^\infty_{S^{2,1}}: \Spc_*^{C_2} \to \Fun^{\Spc^{C_2}_*}((\Spc_*^{C_2})_{S^{2,1}},\Spc_*^{C_2}) \to \Sp^{C_2}, $$
where the first functor is the left adjoint of the real functor evaluating at the tensor unit. This real functor is part of a real adjunction $\Sigma^\infty_{S^{2,1}}\dashv \Omega^\infty_{S^{2,1}}$ that yields a left adjoint real symmetric monoidal functor
$$  \Exc_{S^{2,1}}(C,\Spc_*^{C_2}) \rightleftarrows  \Exc_{S^{2,1}}(C,\Sp^{C_2}).$$
that is an equivalence by \cite[Proposition 10.21]{hsv}.
\end{proof}

\section{Multiplicative structure and universal property}

\subsection{Additivization and real $K$-theory functor}

\begin{definition}Let $\caD $ be a presentable reduced real $\infty$-category. We define a real theory with values in $\caD$ as a reduced real functor $\Wald^\gd_\infty \to \caD$
preserving $\aleph_1$-filtered colimits, finite products and cotensors with $C_2$. \footnote{In \cite{hsv} we call a reduced real functor that preserves finite products and cotensors with $C_2$ \emph{genuine preadditive}.}
\end{definition}

\begin{notation} We denote by $\mathrm{rThy}_\caD$ the full subcategory of real functors $\Wald^\gd_\infty \to \caD$ spanned by the real theories. For $\caD=\Spc^{C_2}_\ast$ we simply write $\mathrm{rThy}$.
\end{notation}

\begin{remark}\label{rmk:rThy_aleph1_description}
There is a canonical equivalence
$$\Fun_{\Spc^{C_2}_*}^{\aleph_1\mathrm{-fcol}} (\Wald^\gd_\infty,\caD)\simeq \Fun_{\Spc^{C_2}_*}((\Wald^\gd_\infty)^{\aleph_1},\caD) $$
where $\Fun_{\Spc^{C_2}_*}^{\aleph_1\mathrm{-fcol}} (\Wald^\gd_\infty,\caD)$ is the full subcategory spanned by the real functors preserving $\aleph_1$-filtered colimits. 

Moreover, under this equivalence the real functors preserving finite products and cotensors with $C_2$ in each of the categories correspond. Hence, we can see $\mathrm{rThy}_\caD$ as a full subcategory of 
$$ \Fun_{\Spc^{C_2}_*}((\Wald^\gd_\infty)^{\aleph_1},\caD) $$
spanned by the real functors preserving finite products and cotensors with $C_2$.

\end{remark}

\begin{remark} A careful reader may wonder, when comparing with the unenriched case (see \cite{barwick.wald}), why we are asking for the preservation of finite products and cotensors with $C_2$. We ask real theories to preserve finite products to ensure that real theories admit a unique lift to connective $C^2$-spectra. This will be later used to ensure that additive theories admit a unique lift to genuine $C^2$-spectra.

Note that in our previous work \cite{hsv}, we had separated the condition of preserving finite products, and had called real theories with that property \emph{preadditive} theories. We think that the current nomenclature presents a better picture of the actual situation and makes the statements cleaner. 
\end{remark}

\begin{proposition}\label{prop:extension_DWald} Let $\caD$ be genuine preadditive. Any real theory $\phi\colon(\Wald^\gd_\infty)^{\aleph_1} \to \caD$ can be uniquely extended to a left adjoint reduced real functor $\phi\colon\D\Wald^\gd_\infty \to \caD$.

\end{proposition}
\begin{proof}
This is a direct consequence of \cref{rmk:rThy_aleph1_description} and \cite[Theorem 3.9.2]{heine2024higher}.
\end{proof}

From now on, we want to focus on those real theories that invert a certain map, which codifies additivity for real $K$-theory; we recall the following construction from  \cite[Section 8.2]{hsv}.

\begin{notation} For any $\caD$ in $\Wald^\gd_\infty$, we denote by $\widetilde{\caD\times\caD}$ the cotensor of $\caD$ with $C_2$ in $\Wald^\gd_\infty$.
\end{notation}

\begin{construction}\label{const:gamma_3}
For every $\caD \in \Wald^\gd_\infty$ the morphism $[1] \simeq \{0 \to 1\} \subset [3]$ in $\Delta$ yields the following map in $\Exact_\infty$
\begin{equation}
\begin{tikzcd}[row sep=tiny]
S(\caD)_3\ar[r]        & S(\caD)_1 \simeq \caD\\
\ (A \to B \to C)\ar[r, mapsto]     &A.
\end{tikzcd}
\end{equation}

On the other hand, the morphism $[1] \simeq \{1 \to 2 \} \subset [3]$ in $\Delta^{hC_2}$ yields a map in $ \Wald^\gd_\infty$ as below
\begin{equation}
\begin{tikzcd}[row sep=tiny]
S(\caD)_3\ar[r]        & S(\caD)_1 \simeq \caD\\
\ (A \to B \to C)\ar[r, mapsto]     & B/A.
\end{tikzcd}
\end{equation}

The adjoint map $ S(\caD)_3 \to \widetilde{\caD \times \caD}$ in $\Wald^\gd_\infty$ of the first map together with the second map give rise to a map in $\Wald^\gd_\infty $

\begin{equation}
\gamma_3\colon S(\caD)_3 \to \widetilde{\caD \times \caD} \times \caD\hspace{0.3cm}\text{ given by }\hspace{0.3cm} \ (A \to B \to C) \mapsto (A, C/B, B/A).
\end{equation}
\end{construction}

\begin{definition}\label{def:pre_semi_additive}Let $\caD$ be a reduced real $\infty$-category that admits reduced cotensors with $S^{1,1}$ and $\phi\colon\Wald_\infty^\gd\to \caD$ a real theory. We say that $\phi$ is \emph{additive} if 
\begin{itemize}
    \item it inverts the map $\gamma_3$ of \cref{const:gamma_3} in $\Wald^\gd_\infty $ for every $\caC \in \Wald^\gd_\infty$, and 
    \item it is genuine excisive, that is, for every $\caC \in \Wald^\gd_\infty$ the natural map $ F(\caC) \to \Omega^{1,1}(F(\Sigma^{1,1}(\caC)))$ is an equivalence.
\end{itemize}
\end{definition}

\begin{notation}
We denote by $$ \Add_\caD \subset\mathrm{rThy}_\caD$$ the full subcategory spanned by the additive theories. For $\caD=\Spc^{C_2}_\ast$ we simply write $\Add$.
\end{notation}

We will build on the relation established in \cite[Section 9.4]{hsv} between the notions above.
\begin{proposition}[{\cite[Theorem 9.32]{hsv}}]\label{add_localization}
The embedding $\Add\hookrightarrow r\mathrm{Thy}$ admits a left adjoint, which we call $\add$.
\end{proposition}

\begin{definition}\label{def:KR_space}
The real $K$-theory functor
$$\KR\colon\Wald_\infty^\gd\to \Spc^{C_2}_\ast$$
is defined as the additivization $\KR\coloneqq\add(\iota)$ of the functor $\iota\colon\Wald_\infty^\gd\to\Spc^{C_2}_\ast$, which sends an $\infty$-category with genuine duality $\caC$ to its maximal subspace in $\caC$ equipped with the restricted genuine duality.
\end{definition}

\begin{remark}
We showed in \cite[Corollary 10.36]{hsv} that real functor $\mathrm{KR}\colon \Wald_\infty^\gd \to \Spc^{C_2}_\ast$ lifts to a real functor $$\mathrm{KR}\colon \Wald_\infty^\gd \to \Sp_{\geq0}^{C_2},$$
for which we use the same notation.
\end{remark}

\subsection{Monoidality of additivization}

Now we set to prove that the localization \cref{add_localization} is monoidal, for which we will use the following results. 

\begin{proposition}\label{prop:subcats_are_monodal_localizations}Let $\caV$ a presentable $\infty$-category, and $\caC$ is a small symmetric monoidal $\caV$-enriched category, and consider $\caP_\caV(\caC)=\Fun_\caV(\caC^\op,\caV)$.
\begin{enumerate}
\item The full subcategory of 
$\caP_{ \caV }(\caC)$
spanned by the finite products preserving presheaves is a symmetric monoidal accessible localization.
\item The full subcategory of 
$ \caP_{ \caV }(\caC)$ 
spanned by the presheaves preserving cotensors with objects of $\caK \subset \caV$ is a symmetric monoidal accessible localization.
\end{enumerate}
\end{proposition}
\begin{proof}

1: The generating local equivalences are of the form : 
$$ \alpha\colon \y(A_1) \coprod \dots \coprod \y(A_n) \to \y(A_1 \coprod \dots \coprod A_n)$$
for $A_1,\dots,A_n \in \caC$ and $n \geq 0.$
We need to see that for every $X \in \caC$
the morphism $\alpha \otimes \y(X)$
is a local equivalence.

But $\alpha \otimes \y(X)$ is
$$ \y(A_1 \otimes X) \coprod \dots \coprod \y(A_n \otimes X ) \to \y(A_1 \otimes X \coprod \dots \coprod A_n \otimes X).$$

2: The generating local equivalences are of the form : 
$$ \beta\colon v \odot y(A) \to \y(v \odot A)$$
for $A \in C$ and $v \in \caK $
We need to see that for every $X \in C$
the morphism $\beta \otimes \y(X)$
is a local equivalence.

But $ \beta \otimes \y(X)$ is
$$ \beta: v \odot y(A \otimes X) \to \y(v \odot (A \otimes X)).$$
\end{proof}

\begin{lemma}\label{lem:pullback_symetric_loc}

Let $\caC$ be a presentably symmetric monoidal real $\infty$-category
and $\caA \subset \caC, \caB \subset \caC$ real symmetric monoidal localizations.
Then also $\caA \times_\caC \caB \subset \caA $ is a real symmetric monoidal localization.
\end{lemma}

\begin{proof}
The lax symmetric monoidal right adjoint of the cobase change $\caA \to \caA \coprod_\caC \caB$ in the $\infty$-category
of presentably symmetric monoidal real $\infty$-categories and left adjoint real symmetric monoidal functors,
is the base change $\caA \times_\caC \caB \subset \caA $
in the $\infty$-category
of presentably symmetric monoidal real $\infty$-categories and right adjoint real lax symmetric monoidal functors.
\end{proof}

\begin{corollary}\label{monoidal_localizations}
The full real subcategory $$\mathrm{rThy} \subset \Fun_{\Spc^{C_2}_\ast}((\Wald_\infty^\gd)^{\aleph_1}, \Spc^{C_2}_\ast)$$ is a symmetric monoidal real accessible localization. In particular, $\mathrm{rThy}$ is a presentably symmetric monoidal real $\infty$-category. 
\end{corollary}

\begin{proof}
This is a direct consequence of \cref{rmk:rThy_aleph1_description} together with \cref{prop:subcats_are_monodal_localizations,lem:pullback_symetric_loc}.
\end{proof}

\begin{construction}
Define $i\colon\widetilde{\caC \times \caC} \times \caC \to S(\caC)_3 $
sending $(A, B, C)$ to the diagram
\begin{equation}\label{diagram1}
\begin{tikzcd}
A \ar[r] & A \oplus C \ar[r] \ar[d] & A \oplus B \oplus C \ar[d] \\
&  C \ar[r] & B \oplus C \ar[d] \\
&& B
\end{tikzcd}
\end{equation}

The duality sends the latter diagram to the diagram
\begin{equation}\label{diagram2}
\begin{tikzcd}
B^\dual \ar[r] & B^\dual \oplus C^\dual \ar[r] \ar[d] & A^\dual \oplus B^\dual \oplus C^\dual \ar[d] \\
&  C^\dual \ar[r] & A^\dual \oplus C^\dual \ar[d] \\
&& A^\dual
\end{tikzcd}
\end{equation}

This is, $i(B^\dual, A^\dual, C^\dual).$ The duality on $\widetilde{\caC \times \caC} \times \caC$ sends 
$(A, B, C) $ to $(B^\dual, A^\dual, C^\dual). $ That is,  $i$ is equivariant.
\end{construction}

\begin{remark}
    Note that  $\gamma_3 \circ i \simeq \id. $
\end{remark}
\begin{construction}\label{constr:theta}
Let $\theta\colon S(\caC)_3 \to S(S(\caC)_3)_3$ be the map in $\Wald_\infty^\gd$
that sends an object $(X \to Y \to Z)$ to
$$ ((X \to X \to X) \to (X \to Y \to Y) \to (X \to Y \to Z)).$$
    
We will show that $\theta$ preserves dualities. Indeed, the dual of  $(X \to Y \to Z)$ in $S(\caC)_3$ is
$$((Z/Y)^\dual \to (Z/X)^\dual \to Z^\dual)$$
which is sent by $\theta$ to
$$ (((Z/Y)^\dual \to (Z/Y)^\dual \to (Z/Y)^\dual) \to ((Z/Y)^\dual \to (Z/X)^\dual \to (Z/X)^\dual) \to ((Z/Y)^\dual \to (Z/X)^\dual \to Z^\dual)).$$

The dual of $$ ((X \to X \to X) \to (X \to Y \to Y) \to (X \to Y \to Z))$$ in $S(S(\caC)_3)_3$ is 
$$ ((0 \to 0 \to Z/Y)^\dual \to (0 \to Y/X \to Z/X)^\dual \to (X \to Y \to Z)^\dual),$$
which is 
$$ (((Z/Y)^\dual \to (Z/Y)^\dual \to (Z/Y)^\dual) \to ((Z/Y)^\dual \to (Z/X)^\dual \to (Z/X)^\dual) \to ((Z/Y)^\dual \to (Z/X)^\dual \to Z^\dual)).$$

Note that $\theta\colon S(\caC)_3 \to S(S(\caC)_3)_3$ is natural in
$\caC \in \Wald_\infty^\gd.$
\end{construction}

\begin{remark}
Note that $S(\caC)_3 \xrightarrow{\theta} S(S(\caC)_3)_3 \xrightarrow{\ev_3} S(\caC)_3 $ is the identity.
\end{remark}

\begin{remark}\label{rmk:Sconstruction_monoidal}
For every $C, D \in \Wald_\infty^\gd$
the canonical map
$$ \Ar([n]) \otimes (C^{\Ar([n])} \otimes D) \simeq (\Ar([n]) \otimes C^{\Ar([n])}) \otimes D \to C \otimes D $$
in $\Wald_\infty^\gd$
corresponds to a map $$ C^{\Ar([n])} \otimes D \to (C \otimes D)^{\Ar([n])} $$
in $\Wald_\infty^\gd$ that restricts to a map 
$$ S_n(C) \otimes D \xrightarrow{\chi} S_n(C \otimes D) $$
in $\Wald_\infty^\gd$.
\end{remark}

\begin{proposition}\label{prop:semiadd_loc_is_monoidal}
The real localization $$ \mathrm{rThy} \longrightarrow \{ \y(\gamma^\caD_3)\mid \caD \in (\Wald_\infty^\gd)^{\aleph_1}\}^{-1} \mathrm{rThy} $$
is symmetric monoidal, where $\y\colon ((\Wald_\infty^\gd)^{\aleph_1})^\op \to \mathrm{rThy} $ denotes the Yoneda embedding.
\end{proposition}

\begin{proof}
We have to see that for every $\caC, \caD \in (\Wald_\infty^\gd)^{\aleph_1}$ the morphism
$$ \gamma^\caC_3 \otimes \caD \colon S(\caC)_3 \otimes D \to (\widetilde{\caC \times \caC} \times \caC) \otimes \caD $$ is a local equivalence, i.e.\ is inverted by the localization.
Since $$\gamma^\caC_3 \circ i^\caC\colon \widetilde{\caC \times \caC} \times \caC \to \widetilde{\caC \times \caC} \times \caC $$ is the identity, the composition below is the identity too $$ (\widetilde{\caC \times \caC} \times \caC) \otimes \caD \xrightarrow{i^{\caC} \otimes \caD} S(\caC)_3 \otimes \caD \xrightarrow{\gamma_3^{\caC} \otimes \caD} (\widetilde{\caC \times \caC} \times \caC) \otimes \caD.  $$

Hence it suffices to show that the other composition $$ S(\caC)_3 \otimes \caD \xrightarrow{\gamma_3^{\caC} \otimes \caD} (\widetilde{\caC \times \caC} \times \caC) \otimes \caD \xrightarrow{i^{\caC} \otimes \caD} S(\caC)_3 \otimes \caD $$
becomes the identity under the localization functor.
We prove this in the following.

The composition 
$$ \gamma_3^{S(\caC)_3 \otimes D} \circ i^{S(\caC)_3 \otimes \caD} $$
is the identity. Since $\gamma_3^{S(\caC)_3 \otimes \caD}$ becomes an equivalence under the localization functor, the composition 
$$ i^{S(\caC)_3 \otimes \caD} \circ \gamma_3^{S(\caC)_3 \otimes \caD}\colon S(S(\caC)_3 \otimes \caD)_3 \xrightarrow{} S(S(\caC)_3 \otimes \caD)_3  $$
becomes the identity under the localization functor.

Hence using $\theta$ from \cref{constr:theta} and $\chi$ from \cref{rmk:Sconstruction_monoidal}, we consider the composite
$$ \xi\colon S(\caC)_3 \otimes \caD \xrightarrow{\theta \otimes \caD} S(S(\caC)_3)_3 \otimes \caD \xrightarrow{\chi} S(S(\caC)_3 \otimes \caD)_3 $$$$ \xrightarrow{i^{S(\caC)_3 \otimes \caD} \circ \gamma_3^{S(\caC)_3 \otimes \caD}} S(S(\caC)_3 \otimes \caD)_3 \xrightarrow{ev_3} S(\caC)_3 \otimes \caD $$
which becomes the canonical map $$ S(\caC)_3 \otimes \caD \xrightarrow{\theta \otimes \caD} S(S(\caC)_3)_3 \otimes \caD \to S(S(\caC)_3 \otimes \caD)_3 \xrightarrow{ev_3} S(\caC)_3 \otimes \caD$$
under the localization functor. The latter map, in turn, factors as $$ S(\caC)_3 \otimes \caD \xrightarrow{\theta \otimes \caD} S(S(\caC)_3)_3 \otimes \caD \xrightarrow{ev_3 \otimes \caD} S(\caC)_3 \otimes \caD, $$ which is the identity.
So $\xi$ becomes the identity under the localization functor.

We identify $\xi$ with the map
$$ S(\caC)_3 \otimes \caD \xrightarrow{\gamma_3^{\caC} \otimes \caD} (\widetilde{\caC \times \caC} \times \caC) \otimes \caD \xrightarrow{i^{\caC} \otimes \caD} S(\caC)_3 \otimes \caD.$$

The map $\xi$ factors as 
$$ S(\caC)_3 \otimes \caD \xrightarrow{\theta \otimes \caD} S(S(\caC)_3)_3 \otimes \caD \xrightarrow{(i^{S(\caC)_3} \circ \gamma_3^{S(\caC)_3}) \otimes \caD} S(S(\caC)_3)_3 \otimes \caD \to S(S(\caC)_3 \otimes \caD)_3 $$$$ \xrightarrow{ev_3} S(\caC)_3 \otimes \caD. $$
The latter map identifies with the map
$$ S(\caC)_3 \otimes \caD \to S(S(\caC)_3)_3 \otimes \caD \xrightarrow{(i^{S(\caC)_3} \circ \gamma_3^{S(\caC)_3}) \otimes \caD} S(S(\caC)_3)_3 \otimes \caD \xrightarrow{ev_3 \otimes \caD} S(\caC)_3 \otimes \caD, $$
which is the tensor product of $\caD$ with the map
$$ S(\caC)_3 \xrightarrow{\theta} S(S(\caC)_3)_3 \xrightarrow{i^{S(\caC)_3} \circ \gamma_3^{S(\caC)_3}} S(S(\caC)_3)_3 \xrightarrow{ev_3} S(\caC)_3.$$
The latter map identifies with the map
$$ S(\caC)_3 \xrightarrow{i^{\caC} \circ \gamma_3^{\caC}} S(\caC)_3 \xrightarrow{\theta} S(S(\caC)_3)_3 \xrightarrow{ev_3} S(\caC)_3,$$
which is the map
$$ i^{\caC} \circ \gamma_3^{\caC}\colon S(\caC)_3 \xrightarrow{} S(\caC)_3.$$

Thus, $\xi$ identifies with the map 
$$ S(\caC)_3 \otimes \caD \xrightarrow{\gamma_3^{\caC} \otimes \caD} (\widetilde{\caC \times \caC} \times \caC) \otimes \caD \xrightarrow{i^{\caC} \otimes \caD} S(\caC)_3 \otimes \caD.$$
\end{proof}

\begin{theorem}
The full subcategory 
$\Add \subset \mathrm{rThy}$ 
is a symmetric monoidal localization.
\end{theorem}

\begin{proof}
This follows from applying \cref{lem:pullback_symetric_loc} to the localizations of \cref{prop:localization_is_monoidal} and \cref{prop:semiadd_loc_is_monoidal}.
\end{proof}

\begin{theorem} The real functor $\KR\colon \Wald_\infty^\gd\to \Spc^{C_2}$ refines to the initial $\mathbb{E}_\infty$-algebra in $\Add$. In other words, the real functor $\KR\colon\Wald_\infty^\gd\to \Spc^{C_2}$ refines to a lax symmetric monoidal real functor, which is initial among 
additive lax symmetric monoidal real functors $\Wald^\gd_\infty\to \Spc^{C_2}$.
\end{theorem}

\begin{proof}
For this, it is enough to show that there is an induced monoidal structure on $\Add $ whose tensor unit is $\KR$. Since $\Add \subset \mathrm{rThy}$ 
is a symmetric monoidal localization, there is an induced presentably symmetric monoidal structure on $\Add$ such that the localization functor
$$ \mathrm{rThy}\to \Add $$
is symmetric monoidal.
This symmetric monoidal localization functor sends the tensor unit of $\mathrm{rThy}$, which is the real functor $\iota\colon \Wald^\gd_\infty\to \Spc^{C_2}_\ast$ that sends each $\caC$ in $\Wald^\gd_\infty$ to the genuine $C_2$-spacce of objects,  to the tensor unit in $\Add$. This concludes the proof, since $\KR$ was defined exactly as the image of $\iota$ under this functor.
\end{proof}

Furthermore, we can upgrade the previous result to functors to genuine $C_2$-spectra.

\begin{theorem}[Universal property]\label{thm:univ_prop_Sp} The real functor $\KR\colon \Wald_\infty^\gd\to \Sp^{C_2}$ refines to the initial $\mathbb{E}_\infty$-algebra in $\Add_{\Sp^{C_2}}$. In other words, the real functor $\KR\colon\Wald_\infty^\gd\to \Sp^{C_2}$ refines to a lax symmetric monoidal real functor, which is initial among 
additive lax symmetric monoidal real functors $\Wald^\gd_\infty\to \Sp^{C_2}$.
\end{theorem}

\begin{proof}
\cref{genspec} implies that there is a symmetric monoidal equivalence
$$ \Add_{\Sp^{C_2}} \simeq \Add.$$
\end{proof}

\begin{remark} 
Since the real functor $\KR$ is initial, one can produce a unique symmetric monoidal trace to any lax symmetric monoidal real functor $\D\Wald^\gd_\infty\to \Sp^{C_2}$.  This opens the opportunity to produce real multiplicative trace maps, which concerns future work of the authors.
\end{remark}

We will finish this work by noting that real $K$-theory preserves $\mathbb{E}_\infty$-algebras. For this, we introduce some notation, we denote by $N\colon\Sp \to \Sp^{C_2}$ the Hill-Hopkins-Ravenel norm, which is symmetric monoidal, and by $\nu\colon\Sp^{C_2} \to \Sp$ the forgetful functor, which is symmetric monoidal, too.
\begin{remark}
The induced functor $\Alg_{\mathbb{E}_\infty}(\Sp) \to \Alg_{\mathbb{E}_\infty}(\Sp^{C_2})$ by $N$
is left adjoint to the functor $\Alg_{\mathbb{E}_\infty}(\Sp^{C_2}) \to \Alg_{\mathbb{E}_\infty}(\Sp)$ induced by $\nu.$
\end{remark}

\begin{proposition}
Let $A$ be an $\mathbb{E}_\infty$-genuine $C_2$-spectrum, and $N(\nu(A)) \to A$
a map in $\Alg_{\mathbb{E}_\infty}(\Sp^{C_2})$ which induces on underlying  $\mathbb{E}_\infty$-ring spectra the codiagonal map.
The $\infty$-category $\Mod_{\nu(A)}(\Sp)^\perf$ of dualizable $A$-module spectra refines canonically to a small stable symmetric monoidal $\infty$-category with genuine duality. 
\end{proposition}

\begin{proof}
We know that $\Mod_{\nu(A)}(\Sp)^\perf$ is a small stable $\infty$-category since $\Sp$ is so, and that it inherits a symmetric monoidal structure from the tensor product in $\Sp$\textemdash the relative tensor product. 

It remains to show that $\Mod_{\nu(A)}(\Sp)^\perf$ admits a genuine duality which is compatible with the symmetric monoidal structure. To construct this, we recall that by \cite[Definition A.52 and Proposition A.53]{HHR_Kervaire} there is a norm $N\colon \Sp \to \Sp^{C_2}$,
which is symmetric monoidal. Hence for every $A \in \Alg_{\mathbb{E}_\infty}(\Sp)$ the induced functor
\begin{equation}\label{map_induced_in_Mod}
\Mod_{\nu(A)}(\Sp) \to \Mod_{N(A)}(\Sp^{C_2}) 
\end{equation} is symmetric monoidal.

Moreover, the functor
$N\colon\Alg_{\mathbb{E}_\infty}(\Sp) \to \Alg_{\mathbb{E}_\infty}(\Sp^{C_2})$ induced by the norm is left adjoint to the functor $\nu\colon\Alg_{\mathbb{E}_\infty}(\Sp^{C_2}) \to \Alg_{\mathbb{E}_\infty}(\Sp)$ induced by restriction $\nu\colon \Sp^{C_2} \to \Sp.$ Besides this adjunction, there is an adjunction $\mathrm{triv}\colon \Sp \rightleftarrows \Sp^{C_2} \colon(-)^{C_2}$ that induces an adjunction
$\mathrm{triv}\colon \Alg_{\mathbb{E}_\infty}(\Sp) \rightleftarrows \Alg_{\mathbb{E}_\infty}(\Sp^{C_2}) \colon(-)^{C_2}$, denoted by the same names.

Since $A$ is a normed spectrum, there is a multiplication $N(\nu(A)) \to A$, which is a map in $\Alg_{\mathbb{E}_\infty}(\Sp^{C_2})$. Thus scalar extension along this map $ N(\nu(A)) \to A $
is a symmetric monoidal functor
$$ \Mod_{N(\nu(A))}(\Sp^{C_2}) \to \Mod_{A}(\Sp^{C_2}).$$

Precomposing with the map in (\ref{map_induced_in_Mod}), we obtain a symmetric monoidal functor
$$ \Mod_{\nu(A)}(\Sp) \to \Mod_{N(\nu(A))}(\Sp^{C_2}) \to \Mod_{A}(\Sp^{C_2})$$
that sends $M\mapsto N(M) \otimes_{N(\nu(A))} A.$
By further composing with the forgetful functor and the functor 
$\Sp^{C_2} \to \Fun([1],\Sp), X \mapsto X^{C_2} \to X^{hC_2}$, we obtain a lax symmetric monoidal functor
$$  \Mod_{\nu(A)}(\Sp) \to \Mod_{N(\nu(A))}(\Sp^{C_2}) \to \Mod_{A}(\Sp^{C_2}) \to \Sp^{C_2} \to \Sp$$
that sends $M\mapsto (M \otimes M \otimes_{A \otimes A} A)^{C_2} \simeq (M \otimes_A M)^{C_2}.$

The symmetric monoidal duality gives a symmetric monoidal equivalence
$$ (\Mod_{\nu(A)}(\Sp)^\perf)^\op  \simeq \Mod_{\nu(A)}(\Sp)^\perf.$$
Composing the latter lax symmetric monoidal functor
$\Mod_{\nu(A)}(\Sp) \to \Sp$ with the symmetric monoidal equivalence
$(\Mod_{\nu(A)}(\Sp)^\perf)^\op  \simeq \Mod_{\nu(A)}(\Sp)^\perf \subset \Mod_{\nu(A)}(\Sp)$
we obtain a lax symmetric monoidal quadratic functor
$$ (\Mod_{\nu(A)}(\Sp)^\perf)^\op \to \Sp$$
that sends $$M\mapsto (M^\dual \otimes_A M^\dual)^{C_2}.$$
This completes the proof.
\end{proof}

\begin{corollary}\label{cor:normed_Einfinty_ringspectrum}
Let $A$ be a normed $\mathbb{E}_\infty$-ring spectrum.
Then  $\KR(A) \coloneqq\KR(\Mod_A(\Sp)^\perf)$ is an $\mathbb{E}_\infty$-ring genuine $C_2$-spectrum. In consequence, we obtain a functor $$ \KR\colon \Alg_{\mathbb{N}_\infty}(\Sp^{C_2}) \to  \Alg_{\mathbb{E}_\infty}(\Sp^{C_2}).$$
\end{corollary}

We highlight below a specific case of the previous results.

\begin{proposition}
Let $A$ be an $\mathbb{E}_\infty$-ring spectrum.
The $\infty$-category $\Mod_A(\Sp)^\perf$ of dualizable $A$-module spectra refines canonically to a small stable symmetric monoidal $\infty$-category with genuine duality. 
\end{proposition}

\begin{proof}
Let $\triv\colon \Sp \to \Sp^{C_2}$ the trivial genuine $C_2$-spectrum functor, which is left adjoint to taking $C_2$-fixed points. We consider the counit map $ N(\nu(A^\triv)) \to A^\triv.$
\end{proof}

\begin{corollary}
Let $A$ be an $\mathbb{E}_\infty$-ring spectrum.
Then  $\KR(A) \coloneqq\KR(\Mod_A(\Sp)^\perf)$ is an $\mathbb{E}_\infty$-ring genuine $C_2$-spectrum. In consequence, we obtain a functor $$ \KR\colon \Alg_{\mathbb{E}_\infty}(\Sp) \to  \Alg_{\mathbb{E}_\infty}(\Sp^{C_2}).$$
\end{corollary}

\section{Multiplicative motives}

In \cite[Theorem 10.8]{hsv} we showed that real $K$-theory admits a factorization
\[
\begin{tikzcd}
\Wald_\infty^\gd\ar[rr, "\KR"]\ar[dr]        &&\Spc^{C_2}\\
&\caM\ar[ur, "\KR'"']  &
\end{tikzcd}
\]
where $\caM$ is a category akin to that of non-commutative motives, and $\KR'$ is the genuine excisive approximation of the canonical real functor $\iota\colon \caM \to \Spc^{C_2}$ that sends $\caC\in \caM$ to its space of objects. 

In this section we recall the construction of $\caM$ and endow it with a symmetric monoidal structure. Furthermore, we study the multiplicative structure of $\KR'$: we show that it is also lax symmetric monoidal, and it enjoys a universal property.

\begin{definition}[{\cite[Definition 8.16., Proposition 8.20.]{hsv}}]
We define the non-abelian derived $\infty$-category of $\Wald_\infty^\gd$, that we denote by  
$$\D\Wald_\infty^\gd \subset \Fun_{\Spc^{C_2}}(((\Wald_\infty^\gd)^{\aleph_1})^\op, \Spc^{C_2})$$ 
as the smallest full subcategory containing the representables and closed under small sifted colimits, where $(\Wald_\infty^\gd)^{\aleph_1}$ denotes the $\aleph_1$-compact objects of $\Wald_\infty^\gd$. 
\end{definition}

We are not ready to introduce the category of motives $\caM$.

\begin{definition}
The real $\infty$-category of fissile Waldhausen $\infty$-categories with genuine duality is the full subcategory 
$$\D\Wald_\infty^{\gd,\fiss} \subset \D\Wald_\infty^{\gd}$$ 
spanned by the real functors $(\Wald_\infty^{\gd})^\op \to \widehat{\Spc}^{C_2}$
that invert the morphisms  $\gamma_3\colon S(\caD)_3 \to \widetilde{\caD \times \caD} \times \caD$ that sends $(A \to B \to C) \mapsto (A, C/B, B/A)$ in $\Wald^\gd_\infty $ for $D \in \Wald^\gd_\infty .$
\end{definition}

\begin{remark} 
Note that in \cite[Proposition 8.20]{hsv} we proved that we could take the full subcategory of $\Fun_{\Spc^{C_2}}(((\Wald_\infty^\gd)^{\aleph_1})^\op, \Spc^{C_2})$ spanned by the real presheaves that preserve finite products and cotensors with $C_2$.
\end{remark}

The key ingredient to establish a monoidal structure on $\D\Wald^\gd_\infty$ is the enriched Day conolution, developed for example in \cite{heine2024higher,HinichDayConvolution}.

\begin{proposition}[{\cite[Corollary 5.3.10.]{heine2024higher}}]\label{Dayco}
Let $\caV$ be a presentably symmetric monoidal\footnote{This means, $\caV$ is presentable and the tensor product commutes with colimits in each variable.} $\infty$-category and 
$\caM $ a small symmetric monoidal $\caV$-enriched $\infty$-category.
The $\caV$-enriched $\infty$-category $\caP_\caV(\caM):= \Fun_\caV(\caM^\op,\caV)$ refines to a presentably symmetric monoidal $\caV$-enriched $\infty$-category.
For every presentably symmetric monoidal $\caV$-enriched $\infty$-category $\caN$ the following induced map is an equivalence
$$\map_{\Calg(\Pr\LMod)}(\Fun_\caV(\caM^\op,\caV),\caN) \to \map_{\Calg(\caV\mathrm{-}\Cat)}(\caM,\caN).$$
\end{proposition}

\begin{corollary}
The symmetric monoidal structure on $\Wald_\infty^\gd$ induces a  symmetric monoidal structure on $\Fun_{\Spc^{C_2}}(((\Wald_\infty^\gd)^{\aleph_1})^\op, \Spc^{C_2})$ given by Day convolution.

\end{corollary}

\begin{proposition}
The presentably monoidal structure on $\Fun_{\Spc^{C_2}}(((\Wald_\infty^\gd)^{\aleph_1})^\op, \Spc^{C_2})$ restricts to a closed symmetric monoidal structure on $\D\Wald_\infty^\gd$.
\end{proposition}

\begin{proof}
This follows since the tensor product of $\Fun_{\Spc^{C_2}}(((\Wald_\infty^\gd)^{\aleph_1})^\op, \Spc^{C_2})$
preserves small colimits componentwise, and restricts to the full subcategory of representables by \cite[Theorem 3.9.2]{heine2024higher}.
\end{proof}

\begin{remark}\label{description_tensor_unit_DWald}
The restricted real symmetric monoidal Yoneda embedding $$ \Wald^\gd_\infty \to \Fun_{\Spc^{C_2}}(((\Wald^\gd_\infty)^{\aleph_1})^\op,\Spc^{C_2})$$
induces a real symmetric monoidal embedding 
$$ \Wald^\gd_\infty \to \D\Wald^\gd_\infty.$$
\end{remark}

\begin{proposition}\label{Dwaldfiss_monoidal}
The symmetric monoidal structure on $\D\Wald_\infty^\gd$ descends to $\D\Wald_\infty^{\gd, \mathrm{fiss}}$.
\end{proposition}

\begin{proof}
The proof is similar to that of \cref{prop:semiadd_loc_is_monoidal}.
\end{proof}

\begin{theorem}\label{motive_lax_sym} The real funcor $\KR'\colon \D\Wald^{\gd,\fiss}_\infty\to \Sp^{C_2}$ is lax symmetric monoidal.
\end{theorem}

\begin{proof}
By \cite[Theorem 10.8]{hsv}, we know that the real functor $\KR$ factors as below
\[
\begin{tikzcd}
\Wald_\infty^\gd\ar[rr, "\KR"]\ar[dr]        &&\Sp^{C_2}\\
&\D\Wald^{\gd,\fiss}_\infty\ar[ur, "\KR'"']  &
\end{tikzcd}
\]

Thus, the functor $\KR'\colon \D\Wald^{\gd,\fiss}_\infty\to \Sp^{C_2}$ factors
as $\D\Wald^{\gd,\fiss}_\infty\to \D\Wald^\gd_\infty$ followed by 
$\KR\colon \D\Wald^\gd_\infty\to \Sp^{C_2}$.
To conclude the proof it's enough then to note that the real functors $\D\Wald^{\gd,\fiss}_\infty\to \D\Wald^\gd_\infty$ and
$\KR\colon \D\Wald^\gd_\infty\to \Sp^{C_2}$ are lax symmetric monoidal.

Let $L\colon \D\Wald^{\gd}_\infty \to \D\Wald^{\gd,\fiss}_\infty$ be the localization real functor.
Then $\KR': \D\Wald^{\gd,\fiss}_\infty \to \Spc^{C_2}$

\end{proof}

\begin{proposition}\label{prop:KR'_initial}The real functor  $\KR'$ is initial among left adjoint and genuine excisive lax symmetric monoidal real functors.
\end{proposition} 

\begin{proof}
    It follows from the initiality of $\KR$ in \cref{thm:univ_prop_Sp},  and the canonical equivalence
    $$ \Fun^{L, \exc}_{\Spc^{C_2}}(\D\Wald^{\gd,\fiss}_\infty,\Sp^{C_2}) \simeq  \Add_{\Sp^{C_2}},$$
    where the superscripts $L$ and $\mathrm{exc}$ refer to left adjoint and excisive respectively.
\end{proof}

\bibliographystyle{alpha}
\bibliography{add}

@article{HHR_Kervaire,
  title={On the nonexistence of elements of {Kervaire} invariant one},
  author={Michael A. Hill and Michael J. Hopkins and Douglas C. Ravenel},
  journal={Annals of Mathematics},
  year={2009},
  volume={184},
  pages={1-262}
}

@article{barwick.mult,
  title={Multiplicative structures on algebraic {K}-theory},
  author={Barwick, Clark},
  journal={Documenta mathematica},
  volume={20},
  pages={859},
  year={2015},
  publisher={European Mathematical Society (EMS)}
}

@Article{barwick.wald,
  Title                    = {On the algebraic {$K$}-theory of higher categories},
  Author                   = {Barwick, Clark},
  Journal                  = {J. Topol.},
  Year                     = {2016},
  Number                   = {1},
  Pages                    = {245--347},
  Volume                   = {9},

  Doi                      = {10.1112/jtopol/jtv042},
  Fjournal                 = {Journal of Topology},
  ISSN                     = {1753-8416},
  Mrclass                  = {18F25 (19D10)},
  Mrnumber                 = {3465850},
  Mrreviewer               = {Christoph Winges},
  Url                      = {https://doi.org/10.1112/jtopol/jtv042}
}

@Article{barwick.exact, 
    title={On exact $\infty$-categories and the {T}heorem of the {H}eart}, volume={151}, DOI={10.1112/S0010437X15007447}, number={11}, journal={Compositio Mathematica}, publisher={London Mathematical Society}, author={Barwick, Clark}, year={2015}, pages={2160–2186}}

@article{elmendorf-mandell,
title = {Rings, modules, and algebras in infinite loop space theory},
journal = {Advances in Mathematics},
volume = {205},
number = {1},
pages = {163-228},
year = {2006},
issn = {0001-8708},
doi = {https://doi.org/10.1016/j.aim.2005.07.007},
url = {https://www.sciencedirect.com/science/article/pii/S0001870805001957},
author = {A.D. Elmendorf and M.A. Mandell}
}

@Article{bgt,
  Title                    = {A universal characterization of higher algebraic {$K$}-theory},
  Author                   = {Blumberg, Andrew J. and Gepner, David and Tabuada,
 Gon{\c{c}}alo},
  Journal                  = {Geom. Topol.},
  Year                     = {2013},
  Number                   = {2},
  Pages                    = {733--838},
  Volume                   = {17},

  Doi                      = {10.2140/gt.2013.17.733},
  Fjournal                 = {Geometry \& Topology},
  ISSN                     = {1465-3060},
  Mrclass                  = {19D10 (18D20 19D25 19D55 55N15 55U40)},
  Mrnumber                 = {3070515},
  Mrreviewer               = {Ross Staffeldt},
  Url                      = {http://dx.doi.org/10.2140/gt.2013.17.733}
}

@article{gurski-johnson-osorno-Ktheory_2cats,
title = {K-theory for 2-categories},
journal = {Advances in Mathematics},
volume = {322},
pages = {378-472},
year = {2017},
issn = {0001-8708},
doi = {https://doi.org/10.1016/j.aim.2017.10.011},
url = {https://www.sciencedirect.com/science/article/pii/S0001870816303942},
author = {Nick Gurski and Niles Johnson and Angélica M. Osorno}
}

@Unpublished{hesselholt-madsen,
  Title                    = {Real algebraic {K}-theory},
  Author                   = {Hesselholt, Lars and Madsen, Ib},
  Note                     = {available at http://www.math.ku.dk/~larsh/papers/s05/}
}

@article{heine2024higher,
  title={The higher algebra of weighted colimits},
  author={Heine, Hadrian},
  journal={arXiv:2406.08925},
  year={2024}
}

@Unpublished{schlichting.derived,
  Title                    = {Hermitian {K}-theory, derived equivalences and {K}aroubi's fundamental theorem},
  Author                   = {Schlichting, Marco},
  Note                     = {available at http://homepages.warwick.ac.uk/~masiap/research/gwdg6Revised1.pdf}
}

@article{Calmes_etal1,
  title={Hermitian {K}-theory for stable $\infty$-categories {I}: Foundations: B. Calm{\`e}s et al.},
  author={Calm{\`e}s, Baptiste and Dotto, Emanuele and Harpaz, Yonatan and Hebestreit, Fabian and Land, Markus and Moi, Kristian and Nardin, Denis and Nikolaus, Thomas and Steimle, Wolfgang},
  journal={Selecta mathematica},
  volume={29},
  number={1},
  pages={10},
  year={2023},
  publisher={Springer}
}

@article {mandell_inverseK,
    AUTHOR = {Mandell, Michael A.},
     TITLE = {An inverse {$K$}-theory functor},
   JOURNAL = {Doc. Math.},
  FJOURNAL = {Documenta Mathematica},
    VOLUME = {15},
      YEAR = {2010},
     PAGES = {765--791},
      ISSN = {1431-0635,1431-0643},
   MRCLASS = {19D23 (18D10 55P42 55P47)},
  MRNUMBER = {2735988},
MRREVIEWER = {Daniel\ A.\ Ramras},
}

@Unpublished{hsv,
  Title                    = {Real {K}-theory for infinity {W}aldhausen categories with genuine duality},
  Author                   = {Heine, Hadrian and Spitzweck, Markus and Verdugo, Paula},
  Note                     = {arXiv:1911.11682},
  Year                     = {2019},
}

@unpublished{comparisonKth,
    AUTHOR = {Heine, Hadrian and Spitzweck, Markus and Verdugo, Paula},
    TITLE = {An equivalence between the real S- and the hermitian Q-construction},
    YEAR = {2024},
    NOTE = {arXiv:2301.2410.07846},
    URL = {https://arxiv.org/abs/2410.07846},
}

@article{HinichDayConvolution,
  title={Colimits in enriched $\infty$-categories and {Day} convolution},
  author={Vladimir Hinich},
  journal={Theory and Applications of Categories},
  year={2023}
}

@book {May_loop_spaces,
    AUTHOR = {May, J. P.},
     TITLE = {The geometry of iterated loop spaces},
    SERIES = {Lecture Notes in Mathematics},
    VOLUME = {Vol. 271},
 PUBLISHER = {Springer-Verlag, Berlin-New York},
      YEAR = {1972},
     PAGES = {viii+175},
   MRCLASS = {55D35},
  MRNUMBER = {420610},
MRREVIEWER = {J.\ Stasheff},
}

@article {Segal_categories_cohomology,
    AUTHOR = {Segal, Graeme},
     TITLE = {Categories and cohomology theories},
   JOURNAL = {Topology},
  FJOURNAL = {Topology. An International Journal of Mathematics},
    VOLUME = {13},
      YEAR = {1974},
     PAGES = {293--312},
      ISSN = {0040-9383},
   MRCLASS = {55B20},
  MRNUMBER = {353298},
MRREVIEWER = {J.\ P.\ May},
       DOI = {10.1016/0040-9383(74)90022-6},
       URL = {https://doi.org/10.1016/0040-9383(74)90022-6},
}

@incollection {inna_multiplicative,
    AUTHOR = {Zakharevich, Inna},
     TITLE = {The category of {W}aldhausen categories is a closed
              multicategory},
 BOOKTITLE = {New directions in homotopy theory},
    SERIES = {Contemp. Math.},
    VOLUME = {707},
     PAGES = {175--194},
 PUBLISHER = {Amer. Math. Soc., [Providence], RI},
      YEAR = {[2018] \copyright 2018},
      ISBN = {978-1-4704-3774-9},
   MRCLASS = {18D99 (18D05 18D10 18D15 18F25)},
  MRNUMBER = {3807747},
MRREVIEWER = {Walter\ Tholen},
       DOI = {10.1090/conm/707/14259},
       URL = {https://doi.org/10.1090/conm/707/14259},
}

\end{document}